\documentclass[11pt]{amsart}
\usepackage{amsmath}
\usepackage{amsxtra}
\usepackage{amscd}
\usepackage{amsthm}
\usepackage{amsfonts}
\usepackage{amssymb}
\usepackage{eucal}
\newcommand{\ket}[1]{{| #1 \rangle}}      

\newcommand{\cA}{\mathcal{A}}

\newcommand{\nn}{\nonumber}
\newcommand{\bea}{\begin{eqnarray}}
\newcommand{\ena}{\end{eqnarray}}
\newcommand{\be}{\begin{eqnarray*}}
\newcommand{\en}{\end{eqnarray*}}
\def\bel{\begin{eqnarray}}
\def\enl{\end{eqnarray}}

\newcommand{\C}{{\mathbb C}}

\newcommand{\Z}{{\mathbb Z}}
\newcommand{\al}{{\alpha}}
\newcommand{\ga}{{\gamma}}
\newcommand{\la}{{\lambda}}
\newcommand{\bla}{\boldsymbol{\lambda}}

\numberwithin{equation}{section}
\newtheorem{thm}{Theorem}[section]
\newtheorem{prop}[thm]{Proposition}
\newtheorem{lem}[thm]{Lemma}

\newcommand{\sln}{\mathfrak{sl}_{n+1}}
\newcommand{\gl}{\mathfrak{gl}}
\newcommand{\Uv}{U_v(\mathfrak{gl}_{n+1})}
\newcommand{\Uvv}{U_{v^{-1}}(\mathfrak{\gl_{n+1}})}
\newcommand{\V}{\mathcal V}
\newcommand{\bV}{\overline{\mathcal{V}}}
\newcommand{\K}{{\mathbb K}}
\newcommand{\one}{{\bf 1}}

\begin{document}

\title[Bosonic formula for the $\sln$ principal subspace]
{Gelfand-Zetlin basis, Whittaker vectors and a 
bosonic formula for the $\sln$ principal subspace}

\author{B. Feigin, M. Jimbo and T. Miwa}
\address{BF: Landau Institute for Theoretical Physics,
Russia, Chernogolovka, 142432, prosp. Akademika Semenova, 1a,   \newline
Higher School of Economics, Russia, Moscow, 
101000,  Myasnitskaya ul., 20 and
\newline
Independent University of Moscow, Russia, Moscow, 119002,
Bol'shoi Vlas'evski per., 11}
\email{bfeigin@gmail.com}
\address{MJ: Department of Mathematics, 
Rikkyo University, Tokyo 171-8501, Japan}
\email{jimbomm@rikkyo.ac.jp}
\address{TM: Department of Mathematics,
Graduate School of Science,
Kyoto University, Kyoto 606-8502,
Japan}\email{tmiwa@kje.biglobe.ne.jp}
\date{\today}

\begin{abstract}
We derive a bosonic formula for 
the character of the principal space in the
level $k$ vacuum module for $\widehat{\mathfrak{sl}}_{n+1}$, 
starting from a known fermionic formula for it. 
In our previous work, the latter was written as a sum 
consisting of Shapovalov scalar products of the Whittaker vectors 
for $U_{v^{\pm1}}(\mathfrak{gl}_{n+1})$. 
In this paper we compute these scalar products 
in the bosonic form, using the decomposition
of the Whittaker vectors in the Gelfand-Zetlin basis. 
We show further that the bosonic formula obtained in this way
is the quasi-classical decomposition of the fermionic formula.
\end{abstract}

\maketitle

\section{Introduction}\label{sec:1}
One of the central results in the theory of Kac-Moody algebras
is the Weyl formula for the characters of the irreducible representations.
This formula can be interpreted ``quasi-classically". It means the following.
Let $L_\chi$ be an integrable representation with the  highest weight $\chi$.
In $L_\chi$ there are some special vectors called
the extremal vectors. They are labelled by the Weyl group and have
the form $w\cdot v_\chi$, where $v_\chi$ is the highest weight vector,
$w$ is an element of the Weyl group and $w\cdot v$ denotes the projective action.
The Weyl formula reads as
\begin{align*}
{\rm ch}\,L_\chi=\sum_{w\in W}C_w(\chi),
\end{align*}
where $C_w(\chi)$ is interpreted as the character of $L_\chi$
in the vicinity of the extremal vector $w\cdot v_\chi$. In this interpretation
we suppose $\chi$ is ``big", so that the extremal vector $w\cdot v_\chi$
is well-separated from other extremal vectors. To be more precise it means 
that, when $\chi\rightarrow\infty$, generically the character $L_\chi$
in the vicinity of $w\cdot v_\chi$ stabilizes and gives $C_w(\chi)$. 
The Weyl formula states that the quasi-classical decomposition is exact
for finite $\chi$. One important point in the decomposition is that each term
$C_w(\chi)$ is, up to a simple monomial, 
the inverse of a (possibly infinite) product of simple factors.

Now suppose that $\hat{\mathfrak g}$
is an affine Kac-Moody algebra, and $L_k$ be the vacuum representation
of level $k$ with the highest weight vector $v_k$. Let
${\mathfrak g}={\mathfrak n}_+\oplus\mathfrak h\oplus{\mathfrak n}_-$
be the Cartan decomposition, and let
$\hat{\mathfrak n}_+={\mathfrak n}_+\otimes\mathbb C[t,t^{-1}]\subset
\hat{\mathfrak g}$ be the nilpotent subalgebra. Set
\begin{align*}
V^k=U(\hat{\mathfrak n}_+)\cdot v_k\subset L_k
\end{align*}
and call it the principal subspace in $L_k$. 
The quasi-classical formula for the character of $V^k$ can also be written.
For example, if ${\mathfrak g}=\mathfrak{sl}_2$, we have \cite{FL}
\begin{align*}
{\rm ch}\,V^k&\buildrel\hbox{\rm\small def}\over ={\rm Tr}_{V^k}q^dz^{H/2}\\
&=\sum_{m=0}^\infty
\frac{q^{km^2}z^{km}}{(q^{2m+1}z)_\infty(q)_m(q^{-2m+1}z^{-1})_m}\,,
\end{align*}
where $d=t\frac d{dt}$ is the scaling operator, $H$ is the generator 
of the Cartan subalgebra, and
$(z)_m=\prod_{i=1}^m(1-q^{i-1}z)$.
In this formula the right hand side is understood as power series in $z$.
The $m=0,1,2,\cdots$ terms are the contributions from the extremal vectors
$v_k$, $e_{-1}{}^kv_k$, $e_{-3}{}^ke_{-1}{}^kv_k,\cdots$.

In general
\begin{align*}
{\rm ch}\,V^k=\sum_{\gamma\in Q^+}C_\gamma(k)
\end{align*}
where $Q$ is the root lattice, and the subset
$Q^+$ consists of linear combinations of 
the simple roots with non-negative integer coefficients.

In \cite{FFJMM}, it was proved that for ${\mathfrak g}=\mathfrak{sl}_3$
\begin{align}
{\rm ch}\,V^k&=\sum_{d_1,d_2=0}^\infty
q^{k(d_1^2+d_2^2-d_1d_2)}z_1^{kd_1}z_2^{kd_2}
J_{d_1,d_2}(q,q^{2d_1-d_2}z_1,q^{2d_2-d_1}z_2),\label{BOS}\\
J_{d_1,d_2}(q,z_1,z_2)&=\frac1{(qz_1)_\infty(qz_2)_\infty(qz_1z_2)_\infty}\nn\\
&\times
\frac{(qz_1^{-1}z_2^{-1})_{d_1+d_2}}
{(q)_{d_1}(q)_{d_2}(qz_1^{-1})_{d_1}(qz_2^{-1})_{d_2}
(qz_1^{-1}z_2^{-1})_{d_1}(qz_1^{-1}z_2^{-1})_{d_2}}.\nn
\end{align}
The terms in this formula are still factorized but they have nontrivial factors
in the numerators. On the other hand, in \cite{FFJMM},
we have also derived another expression for the same character, in which
$J_{d_1,d_2}(q,q^{2d_1-d_2}z_1,q^{2d_2-d_1}z_2)$ is split into 12 terms, each of which
is a simple power in $q,z_1,z_2$ with a factorized denominator.
We call such formula the ``desingularization". In general, $C_\gamma(k)$
is complicated and cannot be factorized.
However, in this paper, we show that at least
a desingularization can be found for the character of the principal subspace
for the vacuum module where $\mathfrak g=\mathfrak{sl}_{n+1}$ (see Theorem \ref{THEREM}
and Proposition \ref{prop:Jd}). For $\mathfrak g=\mathfrak{sl}_3$ we have
\begin{align}\label{SL3}
J_{d_1,d_2}(q,z_1,z_2)
&=\sum_{m=0}^{\min(d_1,d_2)}
(-z_1)^mq^{-m(d_2-m)+m(m-1)/2}
\frac{1}{(q)_{m}(q)_{d_1-m}(q)_{d_2-m}}
\\
&\times
\frac{1}{(q z_1)_{m}(qz_1z_2)_{m}(qz_2)_{d_2-m}(q^{-d_2+m}z_1)_{m}
(q^{-d_2+2m+1}z_1)_{d_1-m}}.\nn
\end{align}

We call such a formula a bosonic formula.
We note that in \cite{FFJMM}  
bosonic formulas for more general modules over $\hat{\mathfrak n}_+$
are obtained in the case where $\mathfrak g=\mathfrak{sl}_3$, in which we 
used more terms than
the case of the vacuum module in this paper.

Following some geometrical ideas from \cite{BrFi}, 
one naturally expects 
that 
in the desingularization of the $\mathfrak{sl}_{n+1}$ formula the terms are 
labelled
by some basis in the Verma modules of $U_v(\mathfrak{gl}_{n+1})$ where $q=v^2$,
actually by the Gelfand-Zetlin basis.
Our proof goes as follows. In \cite{FFJMM2}, we managed to rewrite the 
fermionic formula \cite{FS}
for ${\rm ch}\,V^k$ in terms of the eigenfunctions of the quantum difference 
Toda Hamiltonian.
Such eigenfunctions were written
by using the Whittaker vectors in the Verma modules for $U_v(\mathfrak{gl}_{n+1})$.
In this paper, we decompose the Whittaker vectors in the Gelfand-Zetlin basis.
This decomposition produces the decomposition of the coefficients of the
eigenfunctions. Moreover, each term of this decomposition has a factorized form.
As a by-product we get some interesting fermionic formulas
and their quasi-classical decompositions.

Fermionic formulas are statistical sums over configurations of particles with color and weight.
A configuration of particles is determined by a set of non-negative integers $\mathbf{m}=(m_{i,t})_{(i,t)\in\mathcal S}$
which represents the number of particles with color $i$ and weight $t$.
Given a function $B(\mathbf{m})$, the fermionic sum is of the form
\begin{align*}
F(\mathcal S,B)=\sum_{\mathbf{m}}\frac{q^{B(\mathbf{m})}}{\prod_{(i,t)\in\mathcal S}(q)_{m_{i,t}}}.
\end{align*}
See \eqref{B} for the case we study in this paper. Let us discuss the fermionic formula for the character
${\rm ch}\,V^k$ for $\mathfrak g=\mathfrak{sl}_3$. In this case we take $\mathcal S_k=\{1,2\}\times[1,k]$
for $\mathcal S$.
In \cite{FFJMM2} we have shown that the quasi-classical decomposition is valid in the following sense.
Fix $(m_1,m_2,n_1,n_2)\in\Z_{\geq0}^4$, and consider the above sum with the restriction that
\begin{align*}
\sum_{1\leq t<\hskip-3pt<k}m_{i,t}=m_i,\quad
\sum_{1<\hskip-3pt<t\leq k}m_{i,t}=n_i.
\end{align*}
In the limit $k\rightarrow\infty$ this sum approaches some rational function $F_{1,k}(m_1,m_2,n_1,n_2)$.
In \cite{FFJMM2} we have shown that for finite $k\geq0$, we have the equality
\begin{align*}
F(\mathcal S_k,B)=\sum_{m_1,m_2,n_1,n_2}F_{1,k}(m_1,m_2,n_1,n_2).
\end{align*}
We call this equality the quasi-classical decomposition. In this paper we consider the case where we take
\begin{align*}
\mathcal S_{k',k}=\{(i,t)|1\leq t\leq k\delta_{i,1}+k'\delta_{i,2}\}.
\end{align*}
In the limit $1<\hskip-3pt<k'<\hskip-3pt<k$, we have a similar decomposition:
\begin{align*}
F(\mathcal S_{k',k},B)=\sum_{m_1,m_2,n_1,n_2,l_1}F_{1,k',k}(m_1,m_2,n_1,n_2,l_1).
\end{align*}
The restriction for the sum for $F_{1,k',k}(m_1,m_2,n_1,n_2,l_1)$ is such that
\begin{align*}
\sum_{1\leq t<\hskip-3pt<k'}m_{i,t}=m_i,\quad
\sum_{1<\hskip-3pt<t\leq k'}m_{1,t}+\sum_{k'<t<\hskip-3pt<k}m_{1,t}=n_1,
\sum_{1<\hskip-3pt<t\leq k'}m_{2,t}=n_2,\quad
\sum_{k'<\hskip-3pt<t\leq k}m_{1,t}=l_1.
\end{align*}
There are two remarkable features. 
First, the decomposition is exact for finite $k\geq k'\geq1$. Therefore, if
$k=k'$, it gives another formula for ${\rm ch}\,V^k$. Second, each summand in this decomposition
is factorized. In fact, summing up over $m_1,m_2$ we obtain \eqref{BOS}, \eqref{SL3}.
We will derive such a 
decomposition for general $\mathfrak g=\mathfrak{sl}_{n+1}$
by using the Drinfeld Casimir elements of smaller rank.

Finally, we note that our paper is inspired by \cite{BrFi}.
Actually we study the structure
of the singular points on some moduli spaces by using the equivalent
language from the representation theory of affine Lie algebras.

\section{Whittaker vectors for $\gl_{n+1}$}\label{sec:2}
In this section we recall some 
known facts about Whittaker vectors for $\gl_{n+1}$
and their Shapovalov scalar product,  
including the Toda recursion and fermionic formulas.
We give their explicit formulas 
using the Gelfand-Zetlin basis of Verma modules. 

\subsection{Gelfand-Zetlin basis}\label{sec:GZ}

Throughout the text, we consider the complex Lie algebra $\gl_{n+1}$. 
Let $\epsilon_0,\cdots,\epsilon_n$ be a basis of the Cartan subalgebra
orthonormal with respect to the invariant scalar product $(~,~)$. 
The simple roots and fundamental weights are expressed as 
$\al_i=\epsilon_{i-1}-\epsilon_i$, 
$\omega_i=\epsilon_0+\cdots+\epsilon_{i-1}$, $1\le i\le n$.  
We set $Q=\oplus_{i=1}^n\Z\al_i$, 
$P=\oplus_{i=0}^n\Z\epsilon_i$, and $\rho=\sum_{i=1}^n\omega_i$. 

Let $\Uv$ be the corresponding quantum group over $\K=\C(v)$, 
with generators $\{E_i,F_i\}_{1\le i\le n}$, $\{v^{\pm \epsilon_i}\}_{0\le i\le n}$
and standard defining relations. We set $K_i=v^{\epsilon_{i-1}-\epsilon_i}$. 
For $\la=\sum_{i=0}^n\la_i\epsilon_i \in P$, 
let $\V^\la$ be the Verma module over $\Uv$ generated 
by the highest weight vector $\one^\la$ with defining relations
\begin{equation*}
E_i\one^\la=0 \quad (1\le i\le n),
\qquad v^{\epsilon_i}\one^\la=v^{\la_i}\one^\la
\quad (0\le i\le n).
\end{equation*}

Recall that $\V^\la$ has a distinguished 
basis (known as the Gelfand-Zetlin basis) 
relative to the tower of subalgebras
\begin{equation}
\cA_0\subset \cA_1\subset \cdots \subset \cA_n \,,
\label{tower}
\end{equation}
where 
$\cA_k\simeq U_v(\mathfrak{gl}_{k+1})$ ($k=0,\ldots,n$)   
denotes the subalgebra of $\Uv$ generated by 
$\{E_i,F_i\}_{1\le i\le k}$ and $\{v^{\pm\epsilon_i}\}_{0\le i\le k}$.  
Each subspace of $\V^\la$ which is 
jointly invariant under $\cA_k$'s is one dimensional. 
Such subspaces are labeled by arrays of numbers 
\begin{equation}
\bla=
\begin{matrix}
\lambda_{0,n}&              & \lambda_{1,n} &                   & \cdots & \lambda_{n-1,n} && \lambda_{n,n} \\
    &\lambda_{0,n-1}&           &\lambda_{1,n-1}&& \cdots                &\lambda_{n-1,n-1}& \\
&&\ddots&&\ddots& && \\
&&&\lambda_{0,1}&             &\lambda_{1,1}&& \\
&&&             &\lambda_{0,0}&             && \\
\end{matrix}
\label{GZ} 
\end{equation}
called the Gelfand pattern. 
Here we set 
\begin{align}
\lambda_{k,i}&=\lambda_k-m_{k,i}\,,
\label{hwt1}
\end{align}
and $m_{k,i}$ are non-negative integers satisfying 
\begin{equation}
0=m_{k,n}\le m_{k,n-1}\le m_{k,n-2}\le \cdots\le m_{k,k}
\quad (0\le k\le n)\,.
\label{mij}
\end{equation}
In particular, we have
\begin{align*}
\lambda_{k,n}=\lambda_k.
\end{align*}
For economy of space we shall also write $\bla$ as
\begin{align*}
\bla=(\la^{(n)},\cdots,\la^{(0)}),
\quad
\la^{(i)}=\la_{0,i}\epsilon_0+\cdots+\la_{i,i}\epsilon_i.
\end{align*}

By choosing an appropriate generator 
$\ket{\bla}=\ket{\la^{(n)},\cdots,\la^{(0)}}$  
of each subspace corresponding to \eqref{GZ},  
the action of Chevalley generators can be described explicitly. 
For this purpose it is 
convenient to extend the base field from $\K$ to 
$\mathcal{R}$ obtained by adjoining 
all elements of the form $\sqrt{f}$ ($f\in \K$).  
We use the same symbols $\Uv$ (resp. $\V^\la$) to denote 
$\mathcal{R}\otimes_{\K}\Uv$  
(resp. $\mathcal{R}\otimes_{\K}\V^\la$).  
Then the Chevalley generators act by the formula \cite{J}
\begin{align}
v^{\epsilon_i}\ket{\bla}&=v^{h_i(\bla)-h_{i-1}(\bla)}\ket{\bla}\,,
\label{Chev1}\\
E_i \ket{\bla}&=
\sum_{k=0}^{i-1}c_{k,i-1}(\bla)\ket{\bla^{(k,i-1)}_+}\,,
\label{Chev2}\\
F_i \ket{\bla}&=
\sum_{k=0}^{i-1}c_{k,i-1}(\bla^{(k,i-1)}_-)
\ket{\bla^{(k,i-1)}_-}\,.
\label{Chev3}
\end{align}
Here $h_i(\bla)=\sum_{k=0}^{i}\la_{k,i}$, 
and 
$\bla^{(k,i-1)}_\pm$ signifies the Gelfand pattern  
wherein $\la_{k,i-1}$ is replaced by $\la_{k,i-1}\pm1$ 
while keeping all other $\la_{l,j}$'s unchanged. 
The coefficients $c_{k,i-1}(\bla)$ 
have the factorized form 
\begin{equation*}
c_{k,i-1}(\bla)^2=
-\frac{\prod_{0\le l\le i-2}[\la_{l,i-2}-\la_{k,i-1}-l+k-1]
\prod_{0\le l\le i}[\la_{l,i}-\la_{k,i-1}-l+k]}
{\prod_{0\le l\le i-1\atop l(\ne k)}
[\la_{l,i-1}-\la_{k,i-1}-l+k-1]
[\la_{l,i-1}-\la_{k,i-1}-l+k]}\,.
\end{equation*}
Here and after, we use the symbols $[m]=(v^m-v^{-m})/(v-v^{-1})$, 
$[m]!=[m][m-1]\cdots[1]$, and 
$[m]_k=[m][m+1]\cdots[m+k-1]$.  

\subsection{Whittaker vectors}
The Verma module carries an obvious grading
$\V^\la=\oplus_{\beta\in Q^+}(\V^\la)_\beta$ where
\begin{equation}
(\V^\la)_{\beta}=\{w\in \V^\la\mid
K_i w=v^{(\al_i,\la-\beta)}w\quad(1\le i\le n)
\}\,.
\label{wt1}
\end{equation}
A Whittaker vector 
$\theta^\la=\sum_{\beta\in Q^+}\theta^\la_\beta$ 
is an element of a completion
$\prod_{\beta\in Q^+}(\V^\la)_\beta$ 
of the Verma module. 
It is uniquely 
defined by the conditions that $\theta^\la_0=\one^\la$
and 
\begin{equation}
E_i K_i^{i-1}\ \theta^\la
=\frac{1}{1-v^2}\ \theta^\la
\quad (1\le i\le n)\,.
\label{Whit}
\end{equation}
Let us give an explicit formula for $\theta^\la$ 
in terms of the Gelfand-Zetlin basis. 
For $i=1,\ldots,n$ and parameters 
$\mu=(\mu_0,\cdots,\mu_i)$, 
$\nu=(\nu_0,\cdots,\nu_{i-1})$
satisfying $\mu_k-\nu_k\in\Z_{\ge 0}$, define  
\begin{align}
&A_i(\mu,\nu)^2
=
\label{A}
\\
&\quad
\frac{1}{\prod_{k=0}^{i-1}[\mu_k-\nu_k]!}
\cdot
\frac{1}{\prod_{0\le k< l\le i-1}
[\nu_k-\nu_l-k+l+1]_{\mu_k-\nu_k}}
\cdot
\frac{1}{\prod_{0\le k< l\le i}[\nu_k-\mu_l-k+l]_{\mu_k-\nu_k}}\,.
\nn
\end{align}

\begin{prop}
In the Gelfand-Zetlin basis \eqref{GZ}, 
the Whittaker vector $\theta^\la$ 
has the following representation: 
\begin{equation}
\theta^\la=\sum_{\bla}
\left(\frac{1}{1-v^2}\right)^{ht(\bla)}
\prod_{i=1}^nC_i(\la^{(i)},\la^{(i-1)})\ \ket{\bla}\,.
\label{GTWhit}
\end{equation}
Here we have set 
\begin{align*}
ht(\bla)&=\sum_{0\le k\le i\le n-1}(\la_{k,n}-\la_{k,i}),
\\
C_i(\la^{(i)},\la^{(i-1)})&=v^{p_i(\la^{(i)},\la^{(i-1)})}
A_i(\la^{(i)},\la^{(i-1)})
\\
p_i(\la^{(i)},\la^{(i-1)})
&=(i-1)\left(
\sum_{k=0}^{i-1}\la_{k,i-1}
\Bigl(\sum_{k=0}^{i-1}\la_{k,i-1}
-\sum_{k=0}^{i}\la_{k,i}\Bigr)
-\sum_{0\le k<l\le i-1}\la_{k,i-1}\la_{l,i-1}
+\sum_{0\le k<l\le i}\la_{k,i}\la_{l,i}\right)
\\
&-\sum_{k=1}^{i-1}k(i-k)(\la_{k,i-1}-\la_{k,i})\,.
\end{align*}
The sum ranges over all Gelfand patterns \eqref{GZ}--\eqref{mij}
with 
fixed $\la_0,\ldots,\la_{n}$.  
\end{prop}
\begin{proof}
The proof is a direct calculation using 
formulas  \eqref{Chev1},\eqref{Chev2}
for the action of $E_i,K_i$. 
The defining relations \eqref{Whit}
reduce to the identities
\begin{equation*}
\sum_{l=0}^i\frac{\prod_{k=0}^{i-1}[a_k-b_l]}
{\prod_{k=0\atop k\neq l}^{i}[b_k-b_l]}
\ v^{-\sum_{k=0}^{i-1}a_k+\sum_{k=0\atop k\neq l}^i b_k}=1\,.
\end{equation*}
\end{proof}

\subsection{Scalar product}

The main object of our interest
is the scalar product of the Whittaker and the dual Whittaker
vectors. 
To define the latter, we consider the quantum group $\Uvv$
with parameter $v^{-1}$. Its generators are denoted by 
$\{\bar E_i,\bar F_i\}_{1\le i\le n}$, $\{{\bar v}^{\pm\epsilon_i}\}_{0\le i\le n}$.
Let $\bV^\la$ 
be the Verma module over $\Uvv$ generated by the
highest weight vector $\bar\one^\la$ with defining relations
\begin{equation*}
\bar E_i\bar \one^\la=0 \quad (1\le i\le n), 
\qquad 
{\bar v}^{\epsilon_i} \bar \one^\la=v^{-\la_i}\bar\one^\la
\quad (0\le i\le n).
\end{equation*}
The dual Whittaker vector is defined similarly as an element
$\bar\theta^\la\in 
\prod_{\beta\in Q^+}(\bar\V^\la)_\beta$, 
imposing $\bar\theta^\la_0=\bar\one^\la$ and
\begin{equation}
\bar E_i \bar K_i^{i-1}\ \bar \theta^\la
=\frac{1}{1-v^{-2}}\ \bar \theta^\la
\qquad (1\le i\le n)
\label{DWhit}
\end{equation}
in place of \eqref{Whit}.

Let $\sigma$ be the $\mathcal{R}$-linear 
anti-isomorphism of algebras
given by
\begin{equation}
\sigma:\Uv\to \Uvv,
\quad E_i\mapsto \bar F_i,\
F_i\mapsto \bar E_i,
\ K_i \mapsto \bar K_i^{-1}.
\label{sigma}
\end{equation}
There is a unique non-degenerate 
$\mathcal{R}$-bilinear pairing
$(~,~):\V^\la\times\bV^\la\to\mathcal{R}$ 
such that $(\one^\la,\bar\one^\la)=1$ and
\begin{equation}
(x w, w')=(w,\sigma(x)w')
\label{Shapo}
\end{equation}
for all $x\in\Uv$ and $w\in\V^\la$, $w'\in \bV^\la$.
We call \eqref{Shapo} the Shapovalov pairing. 
The Gelfand-Zetlin basis $\{\ket{\bla}\}$
of $\V^\la$ and $\{\overline{\ket{\bla}}\}$
of $\bV^\la$ 
are orthonormal with respect to the Shapovalov pairing:
$(\ket{\bla}, \overline{\ket{\bla'}})=\delta_{\bla,\bla'}$. 

In \cite{FFJMM2}, we considered the scalar product 
\begin{equation}
J^\la_\beta=J^\la_\beta[0,\infty)
=v^{-(\beta,\beta)/2+(\la,\beta)}\
(\theta^\la_\beta,\bar\theta^\la_\beta).
\label{Jla0}
\end{equation}
We set $J^\la_\beta=0$ unless $\beta\in Q^+$. 
The notation $J^\la_\beta[0,\infty)$ 
comes from the fact that the corresponding fermionic formula
is related to the interval $[0,\infty)$ 
(see Theorem 3.2 in \cite{FFJMM2}
and Proposition \ref{prop:fermi} below).  

In what follows, we choose the variables
\footnote{The present definition for $z_i$ is different from 
\cite{FFJMM2} where $z_i=q^{-(\la,\al_i)}$ was used.}
$z_i=q^{-(\la+\rho,\al_i)}$ and write 
\begin{equation} 
J_{d_1,\cdots,d_n}(q,z_1,\cdots,z_n)
=J^\la_\beta[0,\infty)
\qquad \text{ for $\beta=\sum_{i=1}^nd_i\al_i$}\,.
\label{Jdsec2}
\end{equation} 
These are rational functions in $q$ and $z_1,\cdots,z_n$. 

The explicit formula \eqref{GTWhit} (and for its dual) 
yields the following expression for \eqref{Jdsec2}. 
Set 
\begin{equation*} 
 z_{k,l}=\prod_{j=k+1}^lz_j\,. 
\end{equation*} 

\begin{prop}\label{prop:Jd}
We have 
\begin{align} 
J_{d_1,\cdots,d_n}(q,z_1,\cdots,z_n) 
&=\sum_{m_{0,i-1}+\ldots+m_{i-1,i-1}=d_i\atop 1\le i\le n}
(-1)^{\sum_{i=1}^nd_i-\sum_{i=0}^{n-1}m_{i,i}} 
\label{JdGZ}\\
&\times q^{p(m)} 
\prod_{i=1}^nz_j^{\sum_{k=0}^{j-1}\sum_{i=j+1}^n m_{k,i-1}}
\nn
\\ 
&\times 
\prod_{0\le k<i\le n}\frac{1}{(q)_{m_{k,i-1}-m_{k,i}}} 
\times 
\prod_{0\le k<l<i\le n} 
\frac{1}{(q^{m_{k,i}-m_{l,i-1}}z_{k,l})_{m_{k,i-1}-m_{k,i}}} 
\nn\\ 
&\times 
\prod_{0\le k<l\le i\le n} 
\frac{1}{(q^{m_{k,i}-m_{l,i}+1}z_{k,l})_{m_{k,i-1}-m_{k,i}}} 
\,, 
\nn 
\end{align}  
where     
\begin{equation*}
p(m)=-\sum_{0\le k<l\le i\le n-1}m_{k,i}m_{l,i}    
+\sum_{0\le k<l< i\le n-1}m_{k,i}m_{l,i-1}  
+\frac{1}{2}\sum_{0\le k<i\le n-1}  
m_{k,i}(m_{k,i}-1)\,. 
\end{equation*} 
The sum is taken over all non-negative integers
$m_{k,i}$ satisfying \eqref{mij} and $\sum_{k=0}^{i-1}m_{k,i-1}=d_i$. 
\end{prop}

\noindent{\it Example.} 
We have 
\begin{align*}
n=1&: 
J_{d_1}(q,z_1)=\frac{1}{(q)_{d_1}(q z_1)_{d_1}},
\\
n=2&: 
J_{d_1,d_2}(q,z_1,z_2)
=\sum_{m=0}^{\min(d_1,d_2)}
(-z_1)^mq^{-m(d_2-m)+m(m-1)/2}
\frac{1}{(q)_{m}(q)_{d_1-m}(q)_{d_2-m}}
\\
&\times
\frac{1}{(q z_1)_{m}(qz_1z_2)_{m}(qz_2)_{d_2-m}(q^{-d_2+m}z_1)_{m}
(q^{-d_2+2m+1}z_1)_{d_1-m}}\,.
\end{align*}
The second formula can be further simplified to 
\begin{equation*}
J_{d_1,d_2}(q,z_1,z_2)
=\frac{(q z_1z_2)_{d_1+d_2}}
{(q)_{d_1}(q)_{d_2}(q z_1)_{d_1}(q z_2)_{d_2}
(q z_1z_2)_{d_1}(q z_1z_2)_{d_2}}\,.
\end{equation*}
The existence of a factorized form
is a specific (and rather accidental) feature
of $n=1,2$. It does not hold for $n\ge 3$.
\subsection{Toda Hamiltonian and fermionic formula}
\label{sec:Toda}

The quantity $J_{d_1,\ldots,d_n}(q,z_1,\ldots,z_n)$ admits, 
besides the explicit formula  \eqref{JdGZ},  
other ways of characterization.
For completeness, 
we quote these facts from the literature
adapting to the present notation.

The first is through 
the quantum difference Toda Hamiltonian of type $A$.
It is a $q$-difference operator which acts on functions 
$f(y_1,\cdots,y_n)$:  
\begin{equation}
Hf=\sum_{i=0}^n D_i^{-1}D_{i+1}\bigl(z_{i,n}(1-y_i)f\bigr)\,.
\label{Hamil}
\end{equation}
Here $D_i$ stands for the $q$-shift operator
$(D_if)(y_1,\ldots,y_i,\ldots,y_n)
=f(y_1,\ldots,q y_i,\ldots,y_n)$,  
and we set $y_0=0$, $D_0=D_{n+1}=1$. 

\begin{prop}\label{prop:Toda}\cite{Sev,Et}
The generating series 
\begin{equation}
F(q,y_1,\ldots,y_n;z_1,\ldots,z_n)
=\sum_{d_1,\ldots,d_n\ge0}
J_{d_1,\ldots,d_n}(q,z_1,\ldots,z_n)\
y_1^{d_1}\cdots y_n^{d_n}\,
\label{F}
\end{equation}
is an eigenfunction of the Toda Hamiltonian:
\begin{equation*}
H\ F=\Bigl(\sum_{i=0}^n z_{i,n}\Bigr) F\,.
\end{equation*}
\end{prop}

The second way is the fermionic formula. 
Here we restrict the general consideration 
in \cite{FFJMM2} to the Cartan matrix of type $A$. 
For a (possibly infinite) interval $[r,s]$, 
consider the sum
\footnote{The definition is modified from that of \cite{FFJMM2}, (2.3),
in order to match with the change of the definition of $z_i$.}
\begin{align}
&I_{d_1,\ldots,d_n}(q,z_1,\ldots,z_n|r,s)
\label{fermionic}\\
&\quad
=\sum_{l_{r,i}+\cdots+l_{s,i}=d_i \atop 1\le i\le n}
\frac{
q^{\sum_{t,t'=r}^s \min(t,t')
\left(\sum_{i=1}^nl_{t,i}l_{t',i}
-\sum_{i=1}^{n-1}l_{t,i}l_{t',i+1}
\right)}\prod_{i=1}^n z_i^{\sum_{t=r}^s t l_{t,i}}}
{\prod_{i=1}^n\prod_{t=r}^s(q)_{l_{t,i}}}\,.
\nn
\end{align}

Then we have 
\begin{prop}\label{prop:fermi}\cite{FFJMM2}
The following formula holds.
\begin{equation*}
J_{d_1,\ldots,d_n}(q,z_1,\ldots,z_n)=
I_{d_1,\ldots,d_n}(q,z_1,\ldots,z_n|0,\infty)\,.
\end{equation*}
\end{prop}

\section{Character of the principal subspace}\label{sec:princ}

Consider the affine Lie algebra 
$\widehat{\mathfrak{sl}}_{n+1}
=
\sln[t,t^{-1}]\oplus \C c\oplus \C d$. 
Let $M^k$ be the integrable highest weight vacuum 
module of level $k\in \Z_{\ge0}$. 
Namely $M^k$ is the irreducible highest weight 
$\widehat{\mathfrak{sl}}_{n+1}$-module 
generated by the highest weight vector $w$, such that 
\begin{equation*}
(x\otimes t^j)w=0\qquad (x\in \sln,\ j\ge 0)\,,
\end{equation*}
and the canonical central element $c$ acts as the scalar $k$.
Let $\widehat{\mathfrak{n}}_+=
\mathfrak{n}_+\otimes\C[t,t^{-1}]$ 
be the current algebra over the nilpotent
subalgebra $\mathfrak{n}_+$ of $\sln$. 
The $\widehat{\mathfrak{n}}_+$-submodule generated from $w$, 
\begin{equation*}
V^k=U(\widehat{\mathfrak{n}}_+)w\ \subset M^k
\end{equation*}
is called the principal subspace of $M^k$. 

The following fermionic formula is known \cite{FS} (see also
\cite{FJMMT}). 
\begin{equation*}
\mathop{\rm ch} V^k
=\sum_{l_{t,i}\ge0}
\frac{
q^{\sum_{t,t'=1}^k\min(t,t')
\left(\sum_{i=1}^nl_{t,i}l_{t',i}
-\sum_{i=1}^{n-1}l_{t,i}l_{t',i+1}
\right)}\prod_{i=1}^nz_i^{\sum_{t=1}^k tl_{t,i}}}
{\prod_{i=1}^n\prod_{t=1}^k(q)_{l_{t,i}}}\,.
\end{equation*}
In the notation of \eqref{fermionic}, we have 
\begin{equation*}
\mathop{\rm ch} V^k
=\sum_{d_1,\ldots,d_n\ge 0}
I_{d_1,\ldots,d_n}(z_1,\ldots,z_n|1,k)\,.
\end{equation*}

The main result of the present note is the following bosonic 
formula, which generalizes a result of \cite{FFJMM} for $n=2$. 
\begin{thm}\label{THEREM}
The character of the principal subspace of the 
level $k$ vacuum module over $\widehat{\mathfrak{sl}}_{n+1}$ 
is given by 
\begin{align*}
\mathop{\rm ch} V^k
&=\sum_{d_1,\cdots,d_n\ge 0}
q^{k(\sum_{i=1}^nd_i^2-\sum_{i=1}^{n-1}d_id_{i+1})}
z_1^{kd_1}\cdots z_n^{kd_n}
\\
&\times
\tilde{J}_{d_1,\cdots,d_n}
(q,q^{2d_1-d_2}z_1,q^{-d_1+2d_2-d_3}z_2,\cdots,q^{-d_{n-1}+2d_n}
z_n),
\end{align*}
where 
\begin{equation*}
\tilde{J}_{d_1,\cdots,d_n}(q,z_1,\cdots,z_n)=\frac{1}{\prod_{0\le i<j\le n}
(qz_{i,j})_\infty}\cdot 
J_{d_1,\cdots,d_n}(q,z_1^{-1},\cdots, z_n^{-1})\,, 
\end{equation*}
and $J_{d_1,\cdots,d_n}(q,z_1,\ldots,z_n)$ is given by \eqref{JdGZ}.
\end{thm}
\begin{proof}
Writing $(q)_\beta=(q)_{d_1}\cdots (q)_{d_n}$ for 
$\beta=d_1\al_1+\cdots+d_n\al_n$, let us 
introduce the notation 
\begin{align}
J^\la_\beta[r,s]=
\sum_{\sum_{t=r}^s \gamma_t=\beta}
\frac{q^{(1/2)\sum_{r\le t,t'\le s}\min(t,t')(\gamma_t,\gamma_{t'})
-(\la+\rho,\sum_{t=r}^st\gamma_t)}}{\prod_{t=r}^s(q)_{\gamma_t}}\,.
\label{Jrs}
\end{align}
Then \eqref{fermionic} can be written as
\begin{equation*}
I_{d_1,\ldots,d_n}(q,z_1,\ldots,z_n|r,s)=
J^{\la_1\omega_1+\cdots+\la_n\omega_n}_{d_1\al_1+\cdots+d_n\al_n}[r,s]\,,
\end{equation*}
and hence 
\begin{equation*}
\mathop{\rm ch} V^k=\sum_{\beta\in Q^+}J^\la_\beta[1,k]\,.
\end{equation*}

The following formula was proved in \cite{FFJMM2} (see   
(4.26), Theorem 4.13):
\begin{align*}
J^\la_\beta[0,k]=\sum_{\alpha\in Q^+}
J_\alpha^{\al-\lambda-2\rho}[0,\infty)
J^{\lambda-\alpha}_{\beta-\alpha}[0,\infty)\times
q^{k((\alpha,\alpha)/2-(\lambda+\rho,\alpha))}\,.
\end{align*}
Using the relation 
\begin{equation*}
J^\la_\beta[r+1,s+1]=q^{(\beta,\beta)/2-(\la+\rho,\beta)}
J^\la_\beta[r,s], 
\end{equation*}
we deduce that
\begin{equation}
J^\la_\beta[1,k]=
\sum_{\al\in Q^+}J^{\al-\la-2\rho}_\al[0,\infty)
J^{\la-\al}_{\beta-\al}[1,\infty)
\times
q^{k((\al,\al)/2-(\la+\rho,\al))}\,.
\label{JJJ}
\end{equation}
On the other hand, it is simple to show that 
\begin{equation*}
\sum_{\gamma\in Q^+}J^{\la}_\gamma[1,\infty)
=\frac{1}
{\prod_{0\le i<j\le n}
(q z_{i,j})_\infty}
\,.
\end{equation*}
Summing \eqref{JJJ} over $\beta$, setting $\al=\sum_{i=1}^nd_i\al_i$
 and noting that $q^{-(\la-\al+\rho,\al_i)}=q^{(\al,\al_i)}z_i$, 
we obtain the desired formula. 
\end{proof}
\medskip

\section{Quasi-classical expansion}
\label{sec:quasi}

In this section we extend 
the fermionic formula \eqref{fermionic}
 to the setting 
corresponding to the tower of subalgebras \eqref{tower},  
and discuss its `quasi-classical' decomposition. 
In the following, 
we indicate by suffix $k$ the 
quantities associated with the subalgebra $\mathcal{A}_k\simeq U_v(\gl_{k+1})$:
for instance, $P_k=\oplus_{i=0}^{k}\Z\epsilon_i$ and   
$Q^+_k=\oplus_{i=1}^{k}\Z_{\ge0}\alpha_i$.  

Let 
$-\infty\le r_1\le\cdots\le r_{n}\le\infty$  
be a non-decreasing sequence 
of integers (possibly including $\pm\infty$), and set 
$I=[r_1,\infty)$. 
Generalizing \eqref{Jrs}, 
we define for $\la\in P_{n}$ and $\beta\in Q^+_{n}$
\begin{align}
& J\left({{r_1,r_2} \atop Q^+_1} \Bigl|\cdots
 \Bigl| {{r_{n},\infty} \atop Q^+_{n}}
 \Bigl| \la,\beta\right)
=\sum_{\{\ga_t\}_{t\in I}}
\frac{1}{\prod_{t\in I} (q)_{\ga_t}}q^{B(\{\ga_t\}|\la)},
\label{Fer-tower}\\
&B(\{\ga_t\}|\la)
=\frac{1}{2}\sum_{t,t'\in I}\min(t,t')(\ga_t,\ga_{t'})
-(\la+\rho,\sum_{t\in I} t\ga_t).
\label{B}
\end{align}
The sum in \eqref{Fer-tower} is taken over 
$\ga_t\in Q^+_{n}$ ($t\in I$) such that 
\begin{align*}
&\sum_{t\in I}\ga_t=\beta,\\
&\text{$\ga_t\in Q^+_i$ for $r_i\le t< r_{i+1}$ ($i=1,\cdots, n$)}. 
\end{align*}
In the new notation we have $J^\la_\beta[r,s]=J\left({{r,s}\atop Q^+_n}\Bigl|\la,\beta\right)$. 

Recall that in the completion of $U_v(\mathfrak{gl}_{n+1})$
there is an element $u$ which satisfies
\begin{align*}
K_iu=uK_i,\quad E_iu=u K_i^2 E_i, \quad F_iu=u F_iK_i^{-2}
\qquad \text{for all $i=1,\cdots,n$}. 
\end{align*}
Up to multiplication by a simple factor, 
$u$ is the Drinfeld Casimir element. 
On each weight component $\mathcal{V}_{\beta}^\la$ 
of the Verma module, $u$
acts as the scalar $q^{-(\beta,\beta)/2+(\la+\rho,\beta)}$. 
In \cite{FFJMM2}, the fermionic formula \eqref{fermionic} 
 was derived by inserting $u$ 
in the scalar product \eqref{Jla0} which defines the Whittaker vectors
and calculating it in two different ways. 
The same calculation can be repeated using 
`partial' Drinfeld Casimir element. Namely let  
$u_k$ denote the counterpart of $u$ corresponding to 
the subalgebra $\mathcal{A}_k$, $k=1,\cdots,n$. 

\begin{prop}
Let $r_1\le\cdots\le r_{n}\le 0$, and set $r_{n+1}=0$. 
Then we have
\begin{align}
v^{-(\beta,\beta)/2+(\la,\beta)}
(\prod_{k=1}^n u_k^{-r_k+r_{k+1}}\cdot \theta^\la_\beta,
\bar\theta^\la_\beta)
=
J\left({{r_1,r_2} \atop Q^+_1} \Bigl|
\cdots
 \Bigl|  {{r_{n},\infty} \atop Q^+_{n}}
 \Bigl| \la,\beta
\right). 
\label{partial1}
\end{align}
\end{prop}
\begin{proof}
The calculation is the same as in \cite{FFJMM2}, 
Theorem 3.1, and the proof following it.
\end{proof}

For each $k=1,\cdots,n-1$, the Whittaker vector \eqref{Jla0}
admits the decomposition 
in terms of those for the lower rank subalgebra $\cA_k$:
\begin{align}
&\theta^\la_\beta
=\sum_{\la^{(n-1)},\cdots,\la^{(k)}}
(1-q)^{-\sum_{i=k+1}^n\sum_{l=0}^{i-1}(\la_{l,n}-\la_{l,i-1})-
\sum_{l=0}^k(k-l)(\la_{l,n}-\la_{l,k})}
\label{decomp}\\
&\quad\times 
\prod_{i=k+1}^n C_i(\la^{(i)},\la^{(i-1)})
\cdot
\theta(\la^{(n)},\cdots,\la^{(k)}|\beta^{(k)}),
\nn\\
&
\theta(\la^{(n)},\cdots,\la^{(k)}|\beta^{(k)})
\\
&\quad
=\sum_{\la^{(k-1)},\cdots,\la^{(0)}}
(1-q)^{-\sum_{i=1}^k \sum_{l=0}^{i-1} (\la_{l,k}-\la_{l,i-1})}
\prod_{i=1}^k C_i(\la^{(i)},\la^{(i-1)})\cdot
\ket{\la^{(n)},\cdots,\la^{(0)}}.
\nn
\end{align}
Here $\beta^{(i)}$ are defined by 
\begin{align}
&\beta^{(n)}=\beta,\\
&\left(\la^{(i+1)}-\beta^{(i+1)}\right)\Big|_{P_i}=\la^{(i)}-\beta^{(i)}
  \quad (\beta^{(i)}\in Q^+_i, i=1,\cdots,n-1)\,,
\label{betai}
\end{align}
where $\epsilon_k\big|_{P_i}=\sum_{j=0}^i\delta_{j,k}\epsilon_k$
is the projection to $P_i$.
Note that from \eqref{betai} we see that
 \begin{align*}
\beta^{(i)}-\beta^{(i-1)}=(\lambda_{0,i}-\lambda_{0,i-1})(\alpha_1+\cdots+\alpha_i)
+(\lambda_{1,i}-\lambda_{1,i-1})(\alpha_2+\cdots+\alpha_i)+\cdots+
(\lambda_{i-1,i}-\lambda_{i-1,i-1})\alpha_i.
\end{align*}
Therefore, the sum $\sum_{\la^{(n-1)},\cdots,\la^{(1)}}$ is equivalent to
the sum over partition of $\beta$,
 \begin{align*}
\beta=\gamma^{(1)}+\cdots+\gamma^{(n)}
\end{align*}
where
 \begin{align*}
\gamma^{(i)}=\beta^{(i)}-\beta^{(i-1)}\in
R^+_i\buildrel{\rm\small def}\over=
\Z_{\geq0}(\alpha_1+\cdots+\alpha_i)
\oplus\Z_{\geq0}(\alpha_2+\cdots+\alpha_i)
\oplus
\cdots
\oplus\Z_{\geq0}\alpha_i.
\end{align*}
Note that $\lambda^{(n)}=\lambda$ and other $\lambda^{(i)}$ are determined by
 \begin{align}\label{LAMBDA}
 \lambda^{(i)}=\left(\lambda^{(i+1)}-\gamma^{(i+1)}\right)\Big|_{P_i}.
 \end{align}
Here
$\theta(\la^{(n)},\cdots,\la^{(k)}|\beta^{(k)})$ is a 
weight component of a Whittaker vector 
with respect to the subalgebra $\cA_k\simeq U_v(\gl_{k+1})$
and its Verma module with highest weight $\la^{(k)}$. 
It is so normalized that the coefficient of 
the vector 
$\ket{\la^{(n)},\cdots,\la^{(0)}}$ satisfying 
$\la_{l,k}=\la_{l,k-1}=\cdots=\la_{l,l}$ for all $0\le l\le k-1$
is $1$. 

\begin{lem}\label{lem:JsJ}
Formula \eqref{Fer-tower} can be decomposed as 
\begin{align}
&J\left({{r_1,r_2} \atop Q^+_1} \Bigl|
\cdots
 \Bigl|  {{r_{n},\infty} \atop Q^+_{n}}
 \Bigl| \la,\beta
\right) 
=
\sum_{\la^{(n-1)},\cdots,\la^{(1)}}
J\left({{r_1,\infty} \atop Q^+_1} \Bigl|\la^{(1)},\gamma^{(1)}\right)
\prod_{i=2}^n d(\la^{(i)},\gamma^{(i)}|r_i)A_i(\lambda^{(i)},\lambda^{(i-1)})^2,
\label{Jdsec4}
\end{align}
where $A_i(\mu,\nu)$ is given in \eqref{A}. 
The coefficients $d(\mu,\nu|r)$ have the factorized form 
\begin{align}\label{d} 
d(\mu,\nu|r)=
v^{-(\nu,\nu)/2+(\mu,\nu)}q^{r\left((\nu,\nu)/2-(\mu+\rho,\nu)\right)}
\bigl((1-q)(1-q^{-1})\bigr)^{-(\rho,\nu)}.
\end{align}
\end{lem}
\begin{proof}
The action of $\prod_{k=1}^nu_k^{-r_k+r_{k+1}}$ can be calculated 
by using the decomposition \eqref{decomp} and 
$$
u_k\theta(\la^{(n)},\cdots,\la^{(k)}|\beta^{(k)})
=q^{-(\beta^{(k)},\beta^{(k)})/2+(\la^{(k)}+\rho,\beta^{(k)})}
\theta(\la^{(n)},\cdots,\la^{(k)}|\beta^{(k)}).
$$
Taking the scalar product with $\bar\theta^\la_\beta$ and simplifying the result, 
we obtain the assertion. 
\end{proof}

\begin{lem}\label{lem:dJ}
We have
\begin{align}
J\left({{-\infty,r} \atop Q^+_{n-1}} \Bigl|{{r,\infty}\atop Q^+_n}\Bigl|\lambda^{(n)},\gamma^{(n)}\right)=
d(\la^{(n)},\gamma^{(n)}|r)A_n(\lambda^{(n)},\lambda^{(n-1)})^2.
\label{d-in-J}
\end{align}
\end{lem}
\begin{proof}
Consider the decomposition \eqref{decomp} with $k=n-1$ 
and apply $u_{n-1}^{-r_{n-1}+r_n}u_n^{-r_n}$.  
By the same computation as in the previous Lemma, we find
\begin{align*}
&J\left({{r_{n-1},r_n} \atop Q^+_{n-1}} \Bigl|  {{r_{n},\infty} \atop Q^+_{n}}
 \Bigl|\la^{(n)},\beta^{(n)}\right)\\
&=\sum_{\gamma^{(n)}\in R^+_n,\beta^{(n-1)}=\beta^{(n)}-\gamma^{(n)}\geq0}
J\left({{r_{n-1},\infty} \atop Q^+_{n-1}} \Bigl|\la^{(n-1)},\beta^{(n-1)}\right)
d(\la^{(n)},\gamma^{(n)}|r_n)A_n(\la^{(n)},\la^{(n-1)})^2.
\end{align*}
Now let $r_{n-1}\to-\infty$.  
In this limit, only one term $\beta^{(n-1)}=0$ in the sum 
contributes. 
With this choice the factor $J$ in the right hand side is $1$ and $\beta^{(n)}=\gamma^{(n)}$, 
hence we obtain the desired result. 
\end{proof}

Substituting \eqref{d-in-J} (with $n$ replaced by $i=2,\cdots,n$)
back into \eqref{Jdsec4}, we arrive at the following result.

\begin{thm}\label{thm:quasi}
Notation being as in \eqref{betai}, we have 
\begin{align}
&J\left({{r_1,r_2} \atop Q^+_1} \Bigl|
\cdots
 \Bigl|  {{r_{n},\infty} \atop Q^+_{n}}
 \Bigl| \la,\beta
\right) 
\label{quasi}\\
&=\sum_{\gamma^{(1)}+\cdots+\gamma^{(n)}=\beta,\gamma^{(i)}\in R^+_i}
J\left({{r_1,\infty} \atop Q^+_1} \Bigl| \la^{(1)},\gamma^{(1)}\right)
\times\prod_{i=2}^nJ\left({{-\infty,r_{i}} \atop Q^+_{i-1}} \Bigl|
 {{r_{i},\infty} \atop Q^+_{i}} \Bigl| \la^{(i)},\gamma^{(i)}\right). 
\nn
\end{align}
\end{thm}

This Theorem has the following interpretation. 
In formula \eqref{Fer-tower}, 
let us consider the limiting situation where $r_1\ll \cdots\ll r_{n}$. 
Imagine that we take the sum separately 
over the variables $\ga_t$ taking $t$ to be `in the vicinity' $I_i$ 
of each end point $r_i$, $i=1,\cdots,n$.  
Then the contribution to \eqref{B} would become
\begin{align*}
\sum_{i=1}^{n}B_i-
\sum_{i=1}^{n-1}(\sum_{t\in I_i}t\ga_t,\sum_{j>i}\sum_{t'\in I_j}\ga_{t'})
\end{align*}
where $B_i$ stands for \eqref{B} with $t,t'\in I_i$. 
The corresponding sum, 
with $\sum_{t\in I_i}\ga_t=\gamma^{(i)}$
being fixed,  gives a summand in the right hand side 
of \eqref{quasi}.  
Theorem \ref{thm:quasi} tells that this `quasi-classical decomposition' 
in fact gives an exact answer.

\bigskip

{\it Acknowledgement}.\quad
Research of BF 
is partially supported by SS-3472.2008.2,
Programma RAS, "Elementary particles and Fundamental nuclear physics",
RFBR 08-01-00720,
CNRS-RFBR 09-02-93106 and RFBR 09-01-00242
Research of MJ is supported by the Grant-in-Aid for Scientific
Research B-20340027 and B-20340011.


\bigskip

\begin{thebibliography}{99}

\bibitem[BrFi]{BrFi}
A.~Braverman, M.~Finkelberg, {\it Finite difference quantum Toda lattice
via equivariant $K$-theory,} Transform. Groups
{\bf 10} (2005) 363--386.

\bibitem[Et]{Et}
P.~Etingof,
{\it Whittaker functions on quantum groups and $q$-deformed Toda
operators}, Amer. Math. Soc. Transl. Ser. 2, {\bf 194}
(1999) 9--25.

\bibitem[FFJMM]{FFJMM}
B.~Feigin, E.~Feigin, M.~Jimbo, T.~Miwa, E.~Mukhin,
{\it Principal $\widehat{\mathfrak{sl}_3}$ subspaces and quantum Toda Hamiltonian},
Adv. Stud. Pure Math. , {\bf 54}(2009), 109-166.

\bibitem[FFJMM2]{FFJMM2}
B.~Feigin, E.~Feigin, M.~Jimbo, T.~Miwa, E.~Mukhin,
{\it 
Ferimionic formulas for eigenfunctions of the difference 
Toda Hamiltonian},
Lett. Math. Phys. {\bf 88} (2009) 39--77.

\bibitem[FJMMT]{FJMMT}
B.~Feigin, M.~Jimbo, T.~Miwa, E.~Mukhin and Y.~Takeyama,
{\it Fermionic formulas for $(k, 3)$-admissible configurations},
Publ. RIMS, {\bf 40} (2004) 125--162.

\bibitem[FS]{FS}
B.~Feigin and A.~Stoyanovsky,
{\it Quasi-particle models for the
representations of Lie algebras and geometry of flag manifold},
hep-th/9308079, RIMS preprint {\bf 942};
{\it Functional models for the representations of
current algebras and the semi-infinite Schubert cells},
Funct. Anal. Appl. {\bf 28} (1994) 55--72.

\bibitem[FL]{FL}
B.~Feigin and S.~Loktev,
{\it On finitization of the Gordon identities},
Funct. Anal. Appl. .{\bf 35} (2001) 44--51.

\bibitem[J]{J}
M.~Jimbo,
{\it Quantum R matrix related to the generalized Toda system
:an algebraic approach},
Lecture Notes in Physics, Springer {\bf 246} (1986)  335-361.

\bibitem[Sev]{Sev}
A.~Sevostyanov,
{\it Quantum deformation of Whittaker modules and Toda lattice},
Duke Math. J. {\bf 105} (2000) 211--238.

\end{thebibliography}
\end{document}